\documentclass{amsart}
\usepackage{amssymb}

\newcommand\shf{\sigma}

\newcommand\Bun{\operatorname{\mathfrak{B}un}}
\newcommand\Bunc[1]{\operatorname{\mathfrak{B}un}_c^{#1}}
\newcommand\Man{\operatorname{\mathfrak{M}an}}
\newcommand\IMan{\operatorname{\mathfrak{IM}an}}

\renewcommand{\hat}[1]{\widehat{#1}}

\newcommand\dl{\mathrm{dl}}

\newcommand\B{\operatorname{ {}^b } }

\usepackage[all]{xy}
\usepackage{eurosym, mathrsfs}

\newcommand\boxb[1]{\square_b}

\newcommand\paperintro%
        {%
         }
\newcommand\paperbody%
        {%
         }

\newtheorem{theorem}{Theorem}

\newtheorem{corollary}[theorem]{Corollary}

\newtheorem{lemma}[theorem]{Lemma}
\newtheorem{proposition}[theorem]{Proposition}

\numberwithin{equation}{section}

\theoremstyle{remark}
\newtheorem{definition}{Definition}

\newcommand\cFTs{{}^{\Phi}\overline{T}\kern-1pt{}^*}

\newcommand\even{\mathrm{even}}

\newcommand\Ch{\operatorname{Ch}}

\newcommand\tG{\tilde{G}}

\newcommand\cI{\mathcal{I}}

\newcommand\cM{\mathcal{M}}

\newcommand\cR{\mathcal{R}}

\newcommand\hB{\widehat{B}}
\newcommand\hGB{\widehat{G/B}}

\newcommand\hG{\widehat{G}}

\newcommand\hK{\widehat{K}}

\newcommand\bbC{\mathbb C}

\newcommand\bbP{\mathbb P}

\newcommand\bbR{\mathbb R}
\newcommand\bbS{\mathbb S}

\newcommand\bbZ{\mathbb Z}

\newcommand\CI{{\mathcal{C}}^{\infty}}

\newcommand\cFNs{{}^{\Phi}\overline N\kern-1pt{}^*}

\newcommand\red{\operatorname{red}}

\newcommand\ev{\operatorname{even}}

\newcommand\UU{\operatorname{U}}

\newcommand\ha{\frac{1}{2}}

\newcommand\pa{\partial}

\newcommand\Mon{\text{ on }}
\newcommand\Mor{\text{ or }}

\newcommand\Mst{\text{ s.t. }}

\begin{document}
\title[K-theory and Resolution, I]
{Equivariant K-theory and Resolution\\
I: Abelian actions}

\author{Panagiotis Dimakis and Richard Melrose}
\address{Department of Mathematics, Massachusetts Institute of Technology}
\email{pdimakis@mit.edu, rbm@math.mit.edu}

\begin{abstract}

The smooth action of a compact Lie group on a compact manifold can be
resolved to an iterated space, as made explicit by Pierre Albin and the
second author. On the resolution the lifted action has fixed isotropy type,
in an iterated sense, with connecting fibrations and this structure
descends to a resolution of the quotient. For an abelian group action the
equivariant K-theory can then be described in terms of bundles over the
base with morphisms covering the connecting maps. A similar model is given,
in terms of appropriately twisted deRham forms over the base as an iterated
space, for delocalized equivariant cohomology in the sense of Baum,
Brylinski and MacPherson. This approach allows a direct proof of their
equivariant version of the Atiyah-Hirzebruch isomorphism.

\end{abstract}

\maketitle
\tableofcontents

\paperintro
\section*{Introduction}

One intention of this note is to demonstrate that real blow-up can be an
effective tool in the analysis of smooth group actions, particularly in the
compact case. To do so, we describe equivariant K-theory in terms of
resolved spaces and in consequence introduce (here only in the abelian
case) a geometric model for the delocalized equivariant cohomology of Baum,
Brylinski and MacPherson \cite{MR86g:55006}, designed to realize an
equivariant form of the Atiyah-Hirzebruch isomorphism
\begin{equation}
\Ch:K^*_G(M)\otimes \bbC\overset{\simeq}\longrightarrow H^*_{\dl,G}(M).
\label{EKR.42}\end{equation}
The more general case of the action by a non-abelian compact Lie group will
be treated subsequently. That the non-abelian case is more intricate can be seen from the
computation of the equivariant K-theory in case of an action with single
isotropy type by Wassermann \cite{Wassermann}. See also the paper of Rosu
\cite{MR1970011}. 

Resolution of a group action, as described by Pierre Albin and the second
author in \cite{MR2560748}, replaces it by a tree of actions each with
unique isotropy type and with connecting equivariant fibrations. This
results in a similar resolution of the quotient, which we call an `iterated
space' corresponding to its smooth stratification. The description given
here of the various cohomology theories is directly in terms of smooth
`iterated' objects, bundles or forms, over these iterated spaces with
augmented `pull-back' morphisms covering the connecting
fibrations. Resolution may be thought of as replacing the `analytic
complexity' of strata by the `combinatorial complexity' of iterated
fibrations. The perceived advantage of this is that many standard arguments
can be transferred directly to this iterated setting, since the spaces are
smooth. The objects which appear here have local product structures.

The case of a compact abelian group, $G,$ acting, with single isotopy
group, on a compact manifold (with corners), $M,$ is relatively simple and
forms the core of our iterative approach. If the action is free then each
equivariant bundle is equivariantly isomorphic to the pull-back of a bundle
over the base; equivariant bundles descend to bundles. Equivariant K-theory
is then identified, as a ring, with the ordinary K-theory of the
base. However the structure of $K_G(M)$ as a module over the representation
ring of $G$ is lost in this identification. Tensor product and descent
defines an action of irreducible representations of $G$ on smooth bundles
over the base
\begin{equation}
\shf:\hG\times\Bun(Y)\longrightarrow \Bun(Y)
\label{EKR.44}\end{equation}
which projects to give the action of $\hG$ on $K_G(M).$ In realizing
equivariant K-theory and delocalized equivariant cohomology over the
resolved space we need to retain aspects of $\shf.$

For an abelian action with fixed isotropy group $B\subset G$ there is a
similar reduction to objects on the base. Equivariant bundles may be
decomposed over the dual group, $\hB,$ giving a finite number of
coefficient bundles. Lifting an element of $\hB$ into $\hG$ and taking the
tensor product with the inverse gives the coefficient bundle an action of
$\hGB.$ The case of a principal action then applies and results in a
collection of bundles $W_{\hat g}$ over the base, $Y,$ indexed by $\hat
g\in\hG.$ We assemble these into a bundle over $\hG\times Y$ -- allowed to
have different dimensions over different components -- with two additional
properties. First its support projects to a finite subset of $\hB$ and more
significantly it is `twisted' under the action of $\hGB$ on $\hG$ in the
sense that
\begin{equation}
\shf(\hat h)\otimes W_{\hat h\hat g}= W_{\hat g}.
\label{EKR.43}\end{equation}

In this setting of a single isotropy group, the delocalized equivariant
cohomology is given in terms of a twisted deRham complex. These forms are
given by finite sums of formal products
\begin{equation*}
\sum\limits_{i}\hat g_i\otimes u_i,\ u_i\in\CI(Y;\Lambda ^*)
\label{EKR.58}\end{equation*}
where the twisting law \eqref{EKR.43} is replaced by its cohomological image 
\begin{equation}
(\hat h\hat g)\otimes\Ch(\hat h)\wedge v\simeq \hat g\otimes v,\ \hat h\in\hGB,\
v\in\CI(Y;\Lambda ^*).
\label{EKR.46}\end{equation}
Here $\Ch(\hat h)$ is the Chern character of the bundle, with connection,
given by descent from the representation $\hat h$ interpreted as a trivial
bundle with equivariant action and with product connection. The reduced
bundles may be given connections, consistent with the connection on $\hat
h$ and \eqref{EKR.43} for which the Chern character is a delocalized form
in the sense of \eqref{EKR.46}. For discussions of the equivariant Chern
character see the book \cite{Berline-Getzler-Vergne1} of Berline, Getzler
and Vergne and the paper of Getzler \cite{MR1302758}.

These definitions of reduced bundles and delocalized deRham forms are
extended to iterated objects over the resolution of the quotient, $Y_*,$ by
adding morphisms covering the connecting fibrations.  This leads directly
to the Atiyah-Hirzebruch-Baum-Brylinski-MacPherson isomorphism
\eqref{EKR.42}, proved here using the six-term exact sequences which result
from successive pruning of the isotropy tree.

In outline the paper proceeds as follows. In \S\ref{Res} we recall from
\cite{MR2560748} the resolution $X_*$ of any compact Lie group action on a
compact manifold, with the quotient an iterated space $Y_*.$ The lifting of
equivariant bundles to iterated equivariant bundles on $X_*$ is described
in \S\ref{Lif} and the reduction to twisted iterated bundles over $Y_*$ is
discussed, for abelian actions, in \S\ref{RedA} -- the non-abelian case is
much more intricate because of the appearance of `Mackey twisting'. The
realization of equivariant K-theory in terms of reduced bundles is contained
in \S\ref{EquA} and this leads to the geometric model for delocalized
(abelian) equivariant cohomology in \S\ref{DelA}. The relative sequences
obtained by successive pruning of the isotropy tree are introduced in
\S\ref{CheA} and used to establish \eqref{EKR.42} in \S\ref{Iso}.
Examples of circle actions are considered in \S\ref{Ex}.

The authors are happy to thank Pierre Albin, Victor Guillemin, Eckhard
Meinrenken, Mich\`ele Vergne and David Vogan for helpful conversations.
\paperbody

\section{Resolution}\label{Res}

In \cite{MR2560748}, the resolution of the smooth action of a compact Lie
group on a compact manifold, $M,$ was described. The action stratifies $M,$
into smooth submanifolds, with the isotropy group lying in a fixed
conjugacy class of closed subgroups of $G$ on each stratum. For convenience
we shall assume, without loss of generality, that the quotient, $M/G,$ is
connected. If $M$ is not connected then $G$ acts on the set of components
and we may consider each orbit separately and so assume that $G$ acts
transitively on the set of components. Similarly, we declare the strata,
$M_\alpha\subset M,$ to consist of the images under the action of $G$ of
the individual components of the manifolds where the isotropy class is
fixed. Thus the labelling index, $\alpha\in A,$ records a little more than
the isotropy type.

We recall both the resolution of such a group action and the consequent
resolution of the quotient in terms of `iterated spaces'. This is
essentially the notion of a `resolved stratified space'.

For present purposes the category $\Man$ has as objects the \emph{compact}
manifolds with corners, not necessarily connected. Each such manifold has a
finite collection, $\cM_1(M),$ of boundary hypersurfaces $H\subset M.$ By
definition of a `manifold with corners' we require that these boundary
hypersurfaces are embedded -- they are themselves manifolds with corners
having no boundary faces identified in $M.$ As a result each boundary
hypersurface has a global defining function $0\le\rho\in\CI(M),$ vanishing
simply and precisely on $H.$ As morphisms we will take `smooth interior
b-maps' which is to say smooth maps in the usual sense $M_1\longrightarrow
M_2$ such that the pull-back of a boundary defining function for a boundary
hypersurface of $M_2$ is the product of powers of boundary defining
functions for hypersurfaces of $M_1$ (including the case that the pull-back
is strictly positive). Certainly all smooth diffeomorphisms are interior
b-maps. A smooth $G$ action on $X$ is required to be \emph{boundary free}
in the sense that
\begin{equation}
g\in G,\ H\in\cM_1(M)\Longrightarrow \text{ either }gH=H\Mor gH\cap H=\emptyset.
\label{EKR.11}\end{equation}

In fact the morphisms we are most concerned with here are fibre bundles,
which we call `fibrations'. In this compact context, these are simply the
surjective interior b-maps with surjective differentials. The implicit
function theorem applies to show that for such a map each point in the base
has an open neighbourhood $U$ with inverse image diffeomorphic to the
product $U\times Z$ with $Z$ a fixed (over components of the base) compact
manifold with corners with the map reducing to projection. Note that the
b-map condition is used here; without such an assumption the fibres can be
cut off by boundaries.

\begin{definition}\label{EKR.7} The category, $\IMan,$ has as objects,
  $X_*,$ interated spaces in the following sense. There is a `principal'
  manifold with corners $X=X_0$ which is the root of a tree $X_\alpha$ of
  manifolds corresponding to a partial order (`depth') $\alpha \le\beta\in
  A.$ The boundary hypersurfaces of $X_0$ are partitioned into subsets,
  with elements which do not intersect, forming `collective boundary
  hypersurfaces' $H_\alpha (X_0)\subset \cM_1(X_0).$ These carry fibrations
\begin{equation}
\psi_\alpha :H_\alpha (X_0)\longrightarrow X_\alpha .
\label{EKR.8}\end{equation}
Under the partial order on the $H_\alpha$ two (always collective)
hypersurfaces are related if and only if they intersect and any collection
with non-trivial total intersection forms a chain. For each $\alpha$ the
set of boundary hypersurfaces of $X_\alpha$ is also partitioned into
collective boundary hypersurfaces
\begin{equation}
H_\beta (X_\alpha )=\psi_\alpha (H_\beta ),\ \beta>\alpha 
\label{EKR.9}\end{equation}
and $\psi_\beta$ restricted to $H_\alpha$ factors through a fibration 
\begin{equation}
\psi_{\beta ,\alpha }:H_\beta (X_\alpha )\longrightarrow X_\beta,\ \beta >\alpha ;
\label{EKR.10}\end{equation}
with the base index denoted $0,$ $\psi_\alpha =\psi_{\alpha ,0}.$

A smooth $G$-action on an iterated space is a boundary free $G$ action on each $X_\alpha$
with respect to which all the fibrations $\psi_{\alpha ,\beta}$ are $G$-equivariant.
 \end{definition}

It follows that in an iterated space, for any chain
\begin{equation}
\alpha _1<\alpha _2<\cdots<\alpha _k
\label{Res.3}\end{equation}
there is a sequence of fibrations
\begin{equation}
\xymatrix{
\displaystyle{\bigcap_{1\le j\le k}}H_{\alpha _j}(X_0)
\ar[r]^-{\psi_{\alpha _1}}&
\displaystyle{\bigcap_{2\le k\le k}}H_{\alpha _j}(X_{\alpha _1})
\ar[r]^-{\psi_{\alpha _{2},\alpha _1}}&
\cdots
\ar[r]&
H_{\alpha _k}(X_{\alpha _{k-1}})
\ar[r]^-{\psi_{\alpha _k,\alpha ,{k-1}}}&
X_{\alpha _k}
}
\label{Res.4}\end{equation}
with composite the restriction of $\psi_{\alpha _k}.$ It is also follows
that the fibres of the restricted fibrations have strictly increasing
codimension as submanifolds of the fibres in the hypersurfaces.

Resolution is accomplished in \cite{MR2560748} by radial blow up (which
corresponds to a sequence of interior b-maps) of successive smooth centres
corresponding to the tree of isotropy types, in (any) order of decreasing
codimension. This results in a well-defined iterated space, $X_*,$ with
$G$-action in the sense described above with principal space $X=X_0$
and iterated blow-down map
\begin{equation}
\beta :X\longrightarrow M
\label{E-K-n.335}\end{equation}
giving the resolution of $M.$ The $X_\alpha$ are the resolutions of the
isotropy types $\overline{M_\alpha}$ in the same sense. The important property of the
resolution is that the $G$-action on each (smooth, compact) $X_\alpha$ now
has fixed isotropy type and the `change of isotropy type' occurs within the
fibrations $\psi_{\alpha ,\beta }.$ 

Since the action on each $X_\alpha $ has fixed isotropy type the quotients 
\begin{equation}
Y_\alpha=X_\alpha /G
\label{EKR.12}\end{equation}
are all smooth manifolds with corners having boundary hypersurfaces $H_\beta
(Y_\alpha),$ $\beta >\alpha,$ labelled by the index set 
\begin{equation}
A_\alpha =\{\beta \in A;\beta >\alpha \}
\label{EKR.14}\end{equation}
and forming a tree with the corresponding intersection relations and base
$\alpha .$ The $G$-equivariant fibrations \eqref{EKR.10} descend to give
$Y_*$ the structure of an iterated space
\begin{equation}
\xymatrix{
H_\beta  (X_\alpha )
\ar[r]^{/G}\ar[d]_{\psi_{\beta ,\alpha }}&
H_\beta  (Y_\alpha )
\ar[d]^{\phi_{\beta,\alpha }}\\
X_\beta 
\ar[r]_{/G}&
Y_\beta  
},\ \beta >\alpha,\ \phi_\alpha =\phi_{\alpha ,0}.
\label{Res.6}\end{equation}

\section{Lifting}\label{Lif}

Let $\Bun(M)$ denote the category of finite-dimensional, smooth, complex,
vector bundles over a compact manifold $M,$ with bundle maps as
morphisms. Similarly if $M$ is a smooth $G$-space let $\Bun_G(M)$ denote
the category of bundles with equivariant $G$-action covering the action on
$M$ and with morphisms the bundle maps intertwining the actions. Thus the
equivariant K-theory of $M$ can be realized (see Segal \cite{MR0234452}) as
the Grothendieck group
\begin{equation}
K_G(M)=\Bun_G(M)\ominus\Bun_G(M)/\simeq
\label{EKR.1}\end{equation}
with the relation of stable $G$-equivariant bundle isomorphism. 

In general if $F:M\longrightarrow N$ is a smooth $G$-equivariant map of
$G$-spaces then pull-back defines a functor 
\begin{equation}
F^*:\Bun_G(N)\longrightarrow \Bun_G(M).
\label{EKR.2}\end{equation}
In particular this applies to the blow-down map in the resolution of the
action.

\begin{definition}\label{EKR.3} If $X_*$ is an iterated space we denote by
  $\Bun(X_*)$ the category with objects `iterated bundles' consisting of a
  bundle $B_\alpha\in\Bun(X_\alpha )$ for each $\alpha \in A$ and with pull-back
  isomorphisms specified over each $H_\alpha (X_0),$
\begin{equation}
\mu_\alpha :\phi_\alpha ^*B_\alpha \simeq B_0\big|_{H_\alpha(X_0)}
\label{EKR.4}\end{equation}
which factor through intermediate bundle isomorphisms $\mu_{\alpha,\beta},$
$\alpha <\beta,$ covering the sequence \eqref{Res.4} over each boundary
face of $X_0.$ The morphisms are bundle maps between the corresponding
bundles which commute with the connecting morphisms \eqref{EKR.4}.

If $X_*$ is an iterated space with $G$-action, $\Bun_G(X_*)$ denotes the
category in which the bundles carry $G$-actions covering the actions on the
$X_\alpha$ and the connecting isomorphisms, \eqref{EKR.4}, are
$G$-equivariant; morphisms are then required to be $G$-equivariant.
\end{definition}

\begin{lemma}\label{EKR.5} If the iterated $G$-space $X_*$ is the resolution
  of $M,$ with compact $G$ action, then pull-back under the iterated blow-down map
  defines a functor 
\begin{equation}
\beta^*:\Bun_G(M)\longrightarrow \Bun_G(X_*)
\label{EKR.6}\end{equation}
and every iterated bundle in $\Bun_G(X_*)$ is isomorphic to the image of a
bundle in $\Bun_G(M).$ 
\end{lemma}

\begin{proof} The lifting of the objects, $G$-equivariant bundles, and
  corresponding morphisms under $\beta$ is simply iterated pull-back.  It
  only remains to show that every $G$-equivariant iterated bundle in
  $\Bun_c(X_*)$ is isomorphic to such a pull-back. As shown in
  \cite{MR2560748} the resolution $X_*$ can be `rigidified' by
  choosing product decompositions near all boundary hypersurfaces with
  $G$-invariant smooth defining functions consistent near all corners,
  i.e.\ so that the various retractions commute.

In the simple setting of a compact manifold with boundary, $M,$ suppose $V$
is a smooth vector bundle over $M,$ $U$ is a vector bundle over the
boundary $H$ and $T:V\big|_{H}\longrightarrow U$ is a bundle
isomorphism. Then $V$ can be modified near $H$ to an isomorphic bundle
$\tilde V$ which has fibres over $H$ identified with those of $U$ and
outside a small collar neighbourhood of $H$ has fibres identified with $V.$
This can be accomplished by a rotation in the isomorphism bundle of
$V\oplus U$ and in particular carries over to the equivariant case. Indeed
the standard construction has the virtue of leaving the original bundle
unchanged over any set in the collar over an open set on which $T$ is
already an identification. This allows the bundle isomorphisms to be
`removed' inductively over the isotropy tree.

Once the isomorphisms are reduced to the identity the bundles themselves
can be similarily modified in equivariant collars around the boundary
hypersurfaces of $X_0$ to be constant along the normal fibrations and hence
to be the pull-backs of smooth bundles on the base. Alternatively the
topological bundles obtained by direct projection can be smoothed over $M.$
\end{proof}

Pulling back a $G$-connection from a bundle on $M$ we find:

\begin{corollary}\label{EKR.61} For a $G$-bundle $W_*\in\Bun_G(X_*)$
there are a $G$-equivariant connection on each $W_\alpha$ which are
intertwined by the $\mu_{\alpha ,\beta}$ 
\end{corollary}

Now, we can therefore identify
\begin{equation}
K_G(M)=\Bun_G(X_*)\ominus\Bun_G(X_*)/\simeq
\label{EKR.50}\end{equation}
as the Grothendieck group of iterated $G$-bundles on the resolution up to
stable isomorphism.

Finite dimensional representations of a compact Lie group, $G,$ can be
decomposed into direct sums of tensor products with respect to a fixed set
$\hG$ of irreducibles, which can be identified with the set of
characters. This allows the representation category to be identified with
the $\Bun_c(\hG)$ with objects the finitely supported `bundles' over $\hG$
and morphisms being bundle maps. Here, for a non-connected space, the
objects in $\Bun_c$ are permitted to have different dimensions over
different components but in this case, where there may be infinitely many
components, the bundles must have dimension $0$ outside a compact set. So
the objects consists of a (complex) vector spaces associated to a finite
number of characters. Each object in $\Bun_c(\hG)$ defines an equivariant
bundle over any $G$-space and tensor product with these bundles induces an
action of the representation ring, $R(G)=\hG(\bbZ)$ on $K_G(M).$ Aspects of
this action are particularly important in the sequel.

\begin{proposition}\label{E-K-n.423} For the action of a compact Lie
group on an iterated space $X_*,$ taking the tensor product with a
(finite-dimensional) representation gives a functor
\begin{equation}
\shf:\Bun(\hat G)\times \Bun_G(X_*)\longrightarrow \Bun_G(X_*).
\label{E-K-n.419}\end{equation}
\end{proposition}

\begin{proof} Given an element $(V,E)\in \Bun(\widehat{G})\times
  \Bun_G(X_*)$, the corresponding object in $\Bun_G(X_*)$ is the tensor
  product of $E$ and $V$, with $V$ thought as the trivial iterated bundle over
  $X_*$ with the implied $G$-action. Given an element $V$ of $\Bun(\hat G)$
  and an equivariant iterated bundle map $E\rightarrow F$, we obtain an
  equivariant bundle map $V\otimes E\rightarrow V\otimes F$ and similarly
  morphism of representations $V_1\rightarrow V_2$ and an equivariant
  iterated bundle $E$, we obtain an induced equivariant bundle map
  $V_1\otimes E\rightarrow V_2\otimes E.$

\end{proof}

\section{Reduction}\label{RedA}

The abelian case is considerably simpler than the general one and has been
more widely studied. From this point on, in this paper, we shall assume
that $G$ is compact and abelian. One fundamental simplification is that all
(complex) irreducible representations in the abelian case are
one-dimensional (and of course all 1-dimensional representations are
irreducible). In this case $\hG$ is a discrete abelian group.

As recalled in the Introduction, if a compact Lie group acts freely on a
compact manifold, $X,$ then the quotient, $Y,$ is a compact manifold and $X$ is
a principal bundle over it. For an equivariant bundle over $X$, the action
over each orbit gives descent data for the bundle, defining a vector bundle
over the base. This gives an equivalence of categories
\begin{equation}
\Bun_G(X)\cong \Bun(Y)\text{ if $G$ acts freely.}
\label{E-K-n.418}\end{equation}
For such a free action, tensor product with representations gives a
`quantization' of the dual group
\begin{equation}
\shf:\hG\longrightarrow\Bun(Y)
\label{EKR.23}\end{equation}
corresponding to \eqref{E-K-n.419}.

We need to understand this operation in the more general case of an action
with a fixed isotropy group $B\subset G,$ necessarily a closed
subgroup. There is then a short exact sequence
\begin{equation}
B\longrightarrow G\longrightarrow G/B
\label{E-K-n.427}\end{equation}
that is split since the groups are abelian. The dual sequence
\begin{equation}
 \xymatrix{
\widehat{G/B}\ar[r]& \widehat{G}\ar[r]& \ar@<5pt>[l]^{\tau}\widehat{B}
}
\label{E-K-n.430}\end{equation}
is also exact and split so there exists a group homomorphism $\tau$ as
indicated, giving a right inverse. Two such maps $\tau$, $\tau'$ are
related by a group homomorphism
\begin{equation}
 \mu:\widehat{B}\longrightarrow \widehat{G/B}.
\label{E-K-n.438}\end{equation}
with $\tau'=m(\mu,\tau)$, where $m:\hGB\times\hG\longrightarrow\hG$ is the
multiplication map.

For an action with isotropy group $B,$ the quotient $G/B$ acts freely on
$X$ and the discussion above gives the equivalence of categories and shift functor
\begin{multline}
\Bun_{G/B}(X)\cong\Bun (Y),\ \shf:\widehat{G/B}\longrightarrow
\Bun(Y),\\ Y=X/G=X/(G/B).
\label{E-K-n.439}\end{multline}

It is still the case that $G$-equivariant bundles descend to the quotient
but only after decomposition under the action of $B.$ Consider the space
$\hG\times Y,$ which has a natural action by $\widehat{G/B},$ with $\hat
g\times Y$ mapped to $(\hat h\otimes\hat g)\times Y.$

\begin{definition}\label{EKR.19} For a compact abelian $G$-action on a
  compact manifold $X$ with fixed isotropy group $B\subset G$ and base
  $Y=X/G,$ let $\Bunc B(\hat G\times Y)$ denote the category
  of bundles over $\hG\times Y$ with support which is finite when projected
  to $\hB$ and which satisfy the transformation law 
\begin{equation}
\shf(\hat h)\otimes W_{\hat h\otimes \hat g}=W_{\hat g}\ \forall\ \hat
h\in\widehat{G/B},\ \hat g\in\hG.
\label{EKR.20}\end{equation}
Morphisms are bundle maps over each $\hat g\times Y$ which are natural with
respect to \eqref{EKR.20}.
\end{definition}

Note that we could eliminate the action \eqref{EKR.20} at the expense of
choosing a splitting group homomorphism $\tau:\hB\longrightarrow \hG$ as in
\eqref{E-K-n.430}, reducing elements of $\Bunc B(\hG\times
Y)$ to arbitrary elements of $\Bun_c(\hB\times Y).$  

\begin{proposition}\label{E-K-n.424} For an action of a compact abelian group
with fixed isotropy group $B\subset G$ there is an equivalence of categories
\begin{equation}
R:\Bun_G(X)\cong \Bunc B(\hG\times Y),\ Y=X/G
\label{E-K-n.425}\end{equation}
where $W\in \Bun^B_c(\hG\times Y)$ corresponds to the
$G$-equivariant bundle 
\begin{equation}
\bigoplus_{\hat b\in\hB}\tau(\hat b)\otimes\pi^*(W_{\tau(\hat b)})
\label{EKR.24}\end{equation}
for a splitting homomorphism $\tau$ as in \eqref{E-K-n.430}.%
\end{proposition}

\noindent Elements of $\Bun^B_c(\hG\times Y)$ are our `reduced bundles' in
 this simple case.

In order to define \ref{EKR.24} we to pass to the restriction of an element
of $\Bun^B_c(\hG\times Y)$ to the image $\tau(\hB)\times U$ given by a
splitting homomorphism $\tau.$ As a consequence of \eqref{EKR.20} the result
is independent of the choice of $\tau.$

\begin{proof} The isotropy group acts on the fibres of an equivariant
  bundle $U\in\Bun_G(X)$ which therefore decomposes into a direct sum of
  $B$-equivariant bundles
\begin{equation}
U=\bigoplus_{\hat b\in\hB,\text{finite}}U_{\hat b}
\label{EKR.22}\end{equation}
where the action of $B$ on each term factors through the irreducible
representation $\hat b.$ If $\hat g\in\hG$ is a representation which
restricts to $\hat b$ then the action of $B$ on $\hat g^{-1}\otimes U_{\hat
  b}$ is trivial. This bundle therefore has an equivariant $G/B$-action and
so descends to a bundle $W_{\hat g}$ over $Y.$ Doing this for every $\hat
g$, we define a bundle over $\hG\times Y$, supported over a finite subset
of $\hB.$ Clearly these bundles satisfy \eqref{EKR.20}. Conversely each
element of $\Bun^B_c(\hG\times Y)$ defines an element of $\Bun_G(X).$
\end{proof}

Note that the category $\Bunc B(\hG \times Y)$ is not
determined by the groups and base $Y$ alone since it depends on the `shift'
isomorphism $\shf$ which retains some information about the principal
bundle, namely the images under descent to $Y$ of the trivial $G$-bundles
given by elements of $\widehat{G/B}.$

\section{Reduced K-theory}\label{EquA}

Consider next a principal $G$-bundle, for $G$ compact abelian, and a
$G$-equivariant fibration giving a commutative diagram
\begin{equation}
\xymatrix{
X\ar[r]^{\pi}\ar[d]_{/G}&X_1\ar[d]^{/G}\\
Y\ar[r]_{\pi_1}&Y_1}
\label{E-K-n.432}\end{equation}
where the $G$-action on $X_1$ has fixed isotropy group $B;$ thus $\pi_1$ is
a fibration of smooth compact manifolds. In view of the identification of
equivariant bundles in Proposition~\ref{E-K-n.424}, the pull-back map
descends to an `augmented pull back map'
\begin{equation}
\xymatrix{
\Bun_G(X)\ar@{^(->}[d]&\ar[l]_-{\pi^*}\Bun_G(X_1)\ar@{^(->}[d]\\
\Bun(Y)&\ar[l]_-{\pi_1^{\#}}\Bun^B_c(\hG\times Y_1)
}
\label{E-K-n.433}\end{equation}
given by pull-back followed by summation over a splitting $\tau:\hB\longrightarrow \hG:$
\begin{multline}
\pi_1^{\#}:\Bun^B_c(\hG\times Y_1)\overset{\pi_1^*}\longrightarrow\Bun^B_c(\hG\times
Y)\overset{\shf^{\#}}\longrightarrow \Bun(Y)=\Bun^{\{e\}}_c(\hG\times Y),\\
\shf^{\#}(V)=\bigoplus_{\hat b\in\hB}\shf(\tau(\hat
  b))V_{\tau(\hat b)}.
\label{E-K-n.434}\end{multline}
As implicitly indicated by the notation, $\shf^{\#}(V)$ is independent of
the section $\tau.$

We need this in the more general case of an equivariant fibration between
two actions with fixed isotropy groups. For nested closed subgroups,
$K\subset B\subset G,$ we choose iterated splittings
\begin{equation}
\tau':\widehat{K}\longrightarrow \widehat{B},\
\tau_1:\widehat{B}\longrightarrow \widehat{G}\Longrightarrow
\tau=\tau_1\tau':\widehat{K}\longrightarrow \widehat{G}.
\label{E-K-n.441}\end{equation}
Then 
\begin{equation}
\hat b',\ \hat b\in\widehat{B},\ \hat b'\big|_{K}=\hat
b\Longleftrightarrow\ \exists\ !\ \hat h\in\widehat{G/K}\Mst
\tau_1(\hat b')=\hat h\tau(\hat b). 
\label{EKR.26}\end{equation}
\begin{proposition}\label{E-K-n.435} If \eqref{E-K-n.432} is an equivariant
  fibration between actions of a compact abelian Lie group $G$ with fixed
  isotropy groups $B\supset K$ then pull back of equivariant
  bundles descends to the agumented pull-back map 
\begin{equation}
\pi_1^{\#}:\Bun^B_c(\hG\times Y_1)\longrightarrow
\Bun^K_c(\hG\times Y)
\label{E-K-n.436}\end{equation}
given by pull back on the fibres 
\begin{equation*}
\pi_1^*:\Bun^B_c(\hG\times
Y_1)\longrightarrow \Bun^B_c(\hG\times Y)
\label{EKR.25}\end{equation*}
followed by summation to give the value at the image $\tau(\hat k)$
using \eqref{EKR.26}
\begin{equation}
\begin{gathered}
\left(\shf^{\#}(V)\right)_{\tau(\hat k)}=\bigoplus
\limits_{\{\hat b\in\widehat{B};\hat b\big|_{K}=\hat k\}}
\shf(\hat h)V_{\hat h\tau'(\hat k)},\\ 
\forall\ V\in\Bun^B_c(\hG\times Y),\
\hat k\in\hK.
\end{gathered}
\label{E-K-n.442}\end{equation}
\end{proposition}

\begin{proof} An equivariant fibration can be factored through the fibre
  product 
\begin{equation*}
\tilde X_1=X_1\times_{Y_1}Y\longrightarrow Y,\ X\longrightarrow \tilde
X_1\longrightarrow X_1
\label{EKR.56}\end{equation*}
where the $G$-action on $\tilde X_1$ has isotropy group $B.$ Thus it
suffices to consider the two cases of the pull-back of an action under a
fibration and the quotient of an action with isotropy group $K$ by a larger
subgroup $B.$ In the first case the augmented pull-back is simply the
pull-back as in \eqref{EKR.25} with \eqref{E-K-n.442} being the
  identity. In the second case the base is unchanged, so \eqref{EKR.25} is
  the identity and the summation is over those elements of $\hB$ with fixed
  restriction to $K.$ 
\end{proof}

Now we pass to the general case of the action of a compact abelian Lie
group $G$ on a compact manifold $M$ with resolution $X_*$ and resolved
quotient $Y_*$ as discussed above. The isotropy groups $B_\alpha \subset G$
form a tree with root $B_0$ the principal isotropy group. Generalizing the
choice \eqref{E-K-n.441} we can choose iterative splittings by proceeding
stepwise along chains
\begin{equation}
\tau_{\beta,\alpha }:\widehat{B_\alpha}\longrightarrow
\widehat{B_\beta}\ \forall\ \beta >\alpha ,\
\tau_{\gamma ,\beta }\circ\tau_{\beta,\alpha }=\tau_{\gamma ,\alpha },\ \gamma >\beta >\alpha .
\label{E-K-n.444}\end{equation}
Using notation as for the fibration maps we set $\tau_\alpha =\tau_{\alpha,0}.$
Then the formul\ae\ \eqref{E-K-n.436} and \eqref{E-K-n.442} are valid
for any pair and are consistent along chains.

\begin{definition}\label{EKR.34} Reduced bundles $W_*$ in the case of an
abelian action, consist of the following data
\begin{enumerate}
\item A bundle $W_\alpha \in\Bunc{B_\alpha }(\hG\times Y_*)$ for each element of the tree.
\item For each non-principal isotropy type $\alpha>0$ (so $B_\alpha \supset
  B_0)$ a bundle isomorphism 
\begin{equation}
T_\alpha :\pi_\alpha ^{\#} W_\alpha \simeq W_0\big|_{H_\alpha (Y_0)}.
\label{E-K-n.446}\end{equation}
\item The consistency conditions that for any chain $\alpha _*,$ $\alpha
_k>\cdots>\alpha_1>0$ the isomorphisms \eqref{E-K-n.446} restricted to the
boundary face, of codimension $k,$ 
\begin{equation*}
H_{\alpha _*}(Y_0)=\bigcap_{j}H_{\alpha _j}(Y_0)
\label{EKR.27}\end{equation*}
form a chain, corresponding to isomorphisms for each $\alpha <\beta$
\begin{equation}
T_{\alpha,\beta }:\pi_{\alpha\beta } ^{\#}W_\beta \simeq W_\alpha  \big|_{H_\beta (Y_\alpha)}.
\label{E-K-n.447}\end{equation}
\end{enumerate}

Morphisms between such data consist of bundle maps at each level of the tree
intertwining the isomorphisms $T_\alpha$ in \eqref{E-K-n.446}.
\end{definition}

We denote by $\Bun^{B_*}_c(\hG\times Y_*)$ the category of such
  reduced bundles and the corresponding Grothendieck group of pairs of
  reduced bundles up to stable isomorphism by
\begin{equation}
K_{\red}(Y_*)=\Bun^{B_*}_c(\hG\times
    Y_*)\ominus\Bun^{B_*}_c(\hG\times Y_*)/\simeq.
\label{EKR.51}\end{equation}

\begin{theorem}\label{abelianTh} The equivariant K-theory for the action of
  a compact abelian group on a compact manifold $M$ is naturally identified
  with the reduced K-theory \eqref{EKR.51} of the resolved quotient.
\end{theorem}

\begin{proof} This follows from the equivalence of the categories of
  $G$-equivariant iterated bundles over $X_*$ and reduced bundles over
  $Y_*$ which in turn follows from Propositions~\ref{E-K-n.424} and \ref{E-K-n.435}.
  \end{proof}

\begin{definition}\label{EKR.59} An iterated connection $\nabla_*$ on a
  reduced bundle $W_*\in\Bunc{B_*}(\hG\times Y_*)$ is a connection
    $\nabla_{\hat g,\alpha}$ on each bundle $W_{\hat g,\alpha
    }\in\Bun(Y_\alpha)$ satisfying
\begin{equation}
\nabla_{\hat h}\otimes\nabla_{\hat h\hat g,\alpha}=\nabla_{\hat g,\alpha},\
\hat h\in\widehat{G/B_\alpha }\ \forall\ \alpha \in A,\ \hat g\in\hG
\label{EKR.60}\end{equation}
under the transformation law \eqref{EKR.20} and compatible under agumented
pull-back isomorphisms.

\end{definition}

\begin{lemma}\label{EKR.62} 
Any reduced bundle can be equipped with an iterated connection in the sense
of Definition~\ref{EKR.59}.
\end{lemma}

\begin{proof} Such a connection can be obtained following the reduction
  procedure from a $G$-connection on the corresponding iterated $G$-bundle
  over $X_*.$ It is also straightforward to construct such a connection
  directly.
\end{proof}

The odd version of reduced bundles may be defined by `suspension' -- simply
taking the product with an interval and demanding that all bundles be
trivialized over the end points leading to a category 
\begin{equation}
\Bun^{B_*}_c(\hG\times ([0,1]\times Y_*;(\{0\}\cup\{1\})\times Y_*).
\label{EKR.53}\end{equation}
This leads to the odd version of equivariant K-theory  
\begin{multline}
K_G^1(M)=K^1_{\red}(Y_*)=\\
\Bun^{B_*}_c(\hG\times
    Y_*;(\{0\}\cup\{1\})\times Y_*)\ominus\Bun^{B_*}_c(\hG\times
    Y_*;(\{0\}\cup\{1\})\times Y_*)/\simeq.
\label{EKR.52}\end{multline}

The isotropy tree can also be `pruned' by choosing any subtree 
\begin{equation}
A'\subset A,\ \alpha \in A',\ \beta \in A,\ \beta<\alpha \Longrightarrow
\beta \in A'.
\label{EKR.35}\end{equation}
If $P=A\setminus A'$ is the complement of a tree then reduced bundles which
are trivialized on the elements of $P$ form a subcategory 
\begin{equation}
\Bun^{B_*}_c(\hG\times Y_*;P).
\label{EKR.48}\end{equation}
These correspond to $G$ bundles over $M$ which are trivialized over the
corresponding isotropy types. We denote by $K^*_{\red}(Y_*;P)$ the Grothendieck
groups of these relative spaces of bundles and their suspended versions.

\section{Delocalized equivariant cohomology}\label{DelA}

If $\rho \in\widehat{G}$ is an irreducible representation of a compact
abelian group on a complex line, $E,$ then, the corresponding trivial line bundle
over a $G$-space, $X,$ is $G$-equivariant,
\begin{equation}
E\in\Bun_G(X).
\label{E-K-n.450}\end{equation}
The deRham differential defines a $G$-equivariant connection on
$E.$ If the action of $G$ is free, so $X\longrightarrow Y$ is a
principal $G$-bundle, then $E$ descends to a bundle, $\widetilde{E},$ with
connection. The Chern character therefore defines a multiplicative map
\begin{equation}
\Ch:\widehat{G}\longrightarrow \CI(Y;\Lambda ^{2*}).
\label{E-K-n.452}\end{equation}

Let $R(G)$ be the representation algebra with complex coefficients, so the
vector space of formal finite linear combinations of elements of
$\widehat{G}.$ Then the map \eqref{E-K-n.452} extends to a map of algebras
\begin{equation}
\Ch:R(G)\longrightarrow \CI(Y;\Lambda ^{2*}).
\label{E-K-n.466}\end{equation}
For a closed subgroup, $B\subset G,$ $R(G/B)\longrightarrow R(G)$ gives a
multiplicative action
\begin{equation}
R(G/B)\times R(G)\longrightarrow R(G).
\label{E-K-n.468}\end{equation}
This and \eqref{E-K-n.452}, for $G/B$ lead to:

\begin{definition}\label{EKR.28} For a compact abelian group $G$ acting
  with fixed isotropy group $B$ on a compact manifold $X$ the space of
  twisted forms over the base $Y$ is defined as
\begin{equation}
\CI(Y;\Lambda ^*_{\dl})=\CI(Y;\Lambda ^*)\otimes _{\Ch}R(G).
\label{E-K-n.467}\end{equation}
\end{definition}

Thus an element of this space is a finite linear combination of formal products
$$
u_i\otimes \hat{g}_i,\ u_i\in\CI(Y;\Lambda ^*),\ \hat{g}_i\in\widehat{G}
$$
under the equivalence relation
\begin{equation}
u\otimes \hat{g}\simeq \Ch(\hat h)\wedge u\otimes \hat h\hat
g,\ \forall\ \hat h\in\widehat{G/B},\ u\in\CI(Y;\Lambda ^*).
\label{E-K-n.469}\end{equation}
Since the Chern character is closed, the deRham differential descends 
\begin{equation}
d:\CI(Y;\Lambda ^*_{\dl})\longrightarrow \CI(Y;\Lambda ^*_{\dl}),\ d^2=0.
\label{E-K-n.455}\end{equation}

\begin{lemma}\label{E-K-n.474} Suppose that $\pi_1:X\longrightarrow X_1$ is
  a $G$-equivariant fibration for actions with fixed isotropy groups
  $K\subset B$ and $\tilde{\pi}_1:Y\longrightarrow Y_1$ is the induced
  fibration, then there is a natural augmented pull-back
\begin{equation}
\pi^{\#}_1:\CI(Y_1;\Lambda ^*_{\dl})\longrightarrow \CI(Y;\Lambda ^*_{\dl})
\label{E-K-n.456}\end{equation}
which itertwines the action of $d.$
\end{lemma}

\begin{proof} To define \eqref{E-K-n.456} it suffices to consider three elementary
cases.

First suppose that $\pi$ is simply an isomorphism of principle bundles covering the
identity map $\tilde{\pi}_1.$ The only appearance of the bundle in
\eqref{E-K-n.467}, \eqref{E-K-n.469} is through the Chern character and
this is invariant under such a transformation.

Secondly suppose that $K=B$ but that $\pi_1$ is a $G$-equivariant
fibration. Then, after a bundle isomorphism, this corresponds to $X_1$
being the pull-back of the principal $G/B$ bundle over $Y_1$ under a
fibration $\tilde{\pi}_1.$ The bundles $E$ corresponding to representations
of $G/B$ and their conncections pull back naturally and in this case
\eqref{E-K-n.456} corresponds to the pull-back of the coefficient forms.

Finally then consider the case that $X$ is a principal $G/K$ bundle and
that $K\subset B\subset G$ is a second closed subgroup with
\begin{equation}
\pi_1:X\longrightarrow X_1=X/B,
\label{E-K-n.471}\end{equation}
so $Y=Y_1.$ The equivalence relation \eqref{E-K-n.469}, now for
$\hat i\in\widehat{G/B}$ means that any element of $\CI(X_1;\Lambda
^*_{\dl})$ can be represented by a finite sum  
\begin{equation}
u_i\otimes \hat g_i
\label{E-K-n.472}\end{equation}
where the $\hat g_i\in\hG$ exhaust $\hB$ under
restriction. These can be chosen, and relabelled, to be $\hat f_{kj}\hat
g_{j}$ where the $\hat g_{j}\in\hG$ restrict to exhaust
$\hK$ and $\hat f_{kj}\in\hGB.$ Then 
\begin{equation}
\pi_1^{\#}:\sum\limits_{\text{finite}}u_{jk}\otimes \hat f_{kj}\hat g_{j}=
\sum\limits_{\text{finite}} (\sum\limits_{k}\Ch(f_{kj})^{-1}\wedge u_{jk})\otimes \hat g_j.
\label{E-K-n.473}\end{equation}
For elements of $\widehat{G/B}$ the construction of the Chern character
factors through the projection to $X_1.$

The general case corresponds to a composite of these three cases.
\end{proof}

Our model for the delocalized equivariant cohomology of Baum, Brylinski
and MacPherson in the case of a smooth action of a compact abelian Lie
group $G$ on a compact manifold $M$ is the following data on the resolved quotient.

\begin{definition}\label{E-K-n.458} An element of the delocalized deRham
  complex $\CI(Y_*;\Lambda ^*_{\dl})$ consists of:-
\begin{enumerate}
\item For each $\alpha \in A$ a twisted smooth form $u_\alpha
  \in\CI(Y_\alpha;\Lambda ^*_{\dl}).$
\item Compatibility conditions at all boundary faces
\begin{equation}
u_\alpha \big|_{H_\beta  (Y_\alpha )}=\pi_{\alpha\beta }^{\#}u_\beta ,\ \beta >\alpha .
\label{E-K-n.475}\end{equation}
including the  boundary hypersurfaces of the principal quotient
corresponding to $\alpha =0.$
\end{enumerate}
\end{definition}

Again the relative versions corresponding to a subtree $A'\subset A,$
$A=A'\sqcup P$ are similarly defined by demanding that the forms vanish
over the boundary hypersurfaces indexed by $P.$ 

If $\nabla_*$ is an iterated connection on an iterated bundle
$W_*\in\Bunc{B_*}(\hG\times Y_*),$ as in Definition~\ref{EKR.59} and
Lemma~\ref{EKR.62} then the Chern character of each bundle $W_\alpha$ is a
form on on $\hG\times Y_\alpha:$
\begin{equation}
\Ch(W_\alpha ,\nabla_\alpha )_{\hat g}=
\Ch((W_\alpha)_{\hat g} ,\nabla_\alpha) \Mon\{\hat g\}\times Y_\alpha .
\label{EKR.57}\end{equation}
\begin{proposition}\label{EKR.49} The Chern character of a reduced bundle
with compatible connection is an element of $\CI(Y_*;\Lambda^*_{\dl}).$
\end{proposition}

\begin{proof} The forms \eqref{EKR.57} shift correcly under the action of
  $\hGB$ in view of the corresponding property for the connections and the
  iterative relations over the boundary fibrations similarly follow from
  the standard properties of the Chern character under pull-back.
\end{proof}

The verification of the `Atiyah-Hirzebruch-Baum-Brylinski-MacPherson'
isomorphism \eqref{E-K-n.463} is given in \S\ref{Iso} below. The iterative
proof uses the six-term exact sequences arising from pruning the isotropy
tree at successive levels. As with the whole approach here, this is based on
reduction to the case of a fixed isotropy group where the result reduces in
essence to the Atiyah-Hirzebruch isomorphism.

\begin{proposition}\label{EKR.29} If $G$ is a compact abelian group acting
  on a compact manifold with fixed isotropy group $B$ then the Chern
  character gives an isomorphism of $K_G(M)\otimes \bbC$ and $H^{\even}_{\dl}(Y).$
\end{proposition}

\begin{proof} The Atiyah-Hirzebruch isomorphism is valid rationally. This
  amounts to the two statements that for a compact manifold (with corners)
 the range of the Chern character
\begin{equation}
\Ch:K(Y)\longrightarrow H^{\ev}(Y)
\label{EKR.31}\end{equation}
spans the cohomology (with complex coefficients) and that the null space
consists of torsion elements. At the bundle level this means that if the
Chern character for a pair of bundles $V_+\ominus V_-$ is exact then there
for some integers $p$ and $N$  
\begin{equation}
I:V_+^p\oplus\bbC^N\longrightarrow V_-^p\oplus\bbC^N.
\label{EKR.32}\end{equation}
A given connection on the $V_\pm$ lifts to a connection which can then be
deformed to commute with $I$ and so have zero chern character.

Now, in the equivariant case we can consider a splitting homomorphism
$\tau:\hB\longrightarrow \hG$ and then pull a pairs of bundles
$V_{\pm}\in\Bunc B(\tG\times Y) $ back to $\hB\times Y$ where the Chern
character is given by
\begin{equation}
\sum\limits_{\hat b\in\hB}\tau(\hat g) \otimes(\Ch(V_{+,\tau(\hat b)}-\Ch(V_{-,\tau(\hat b)}).
\label{EKR.33}\end{equation}
The vanishing of the class $H^{\ev}(Y;\Lambda _{\dl})$ is equivalent to the
exactness of each of the deRham classes $\Ch(V_{+,\tau(\hat
  b)}-\Ch(V_{-,\tau(\hat b)}.$ Thus the vanishing of the Chern character
in delocalized cohomology implies that each of the pairs $V_{\pm,\tau(\hat
  b)}$ is stably trivial in the sense of \eqref{EKR.32}.

Since $\hB$ is finite and we may always further stabilize \eqref{EKR.32} by
taking powers and adding trivial bundles, we may take $p$ to be the product
of the integers for each $\hat b$ and similarly increase $N.$ This however
amounts to a stable trivialization of the whole bundle $V_+^p\ominus V_-^p$ as an
element $\Bun^B_c(\hG\times Y)$ and proves the injectivity of
\eqref{E-K-n.463} in this case.

The surjectivity is a direct consequence of Atiyah-Hirzebruch isomorphism
and the definition of delocalized forms.
\end{proof}

\section{The relative sequences}\label{CheA}

Our proof of Theorem~\ref{E-K-n.460} is based on induction over 
pruning and the six-term exact sequences which results from passing from one
subtree to another with one more element
\begin{equation}
A''=A'\cup\{\alpha\},\ \alpha \notin A',\ \beta <\alpha \Longrightarrow \beta
\in A'.
\label{EKR.37}\end{equation}
Reduced bundles `supported' on subtrees and the corresponding K-groups
are discussed above.

\begin{proposition}\label{EKR.55} For any subtrees $A'=A\setminus P$
  and $A'\setminus\{\alpha\}$ there is a six term exact sequence 
\begin{equation}
\xymatrix{
K^0_{\red}(Y_*;P\cup\{\alpha \})\ar[r]&K^0_{\red}(Y_*;P)\ar[r]&K^0_{\red}(Y_\alpha ;P)\ar[d]\\
K^1_{\red}(Y_\alpha ;P)\ar[u]&K^1_{\red}(Y_*;P)\ar[l]&K^1_{\red}(Y_*;P\cup\{\alpha \}).\ar[l]
}
\label{10.7.2018.3}\end{equation}
\end{proposition}

\begin{proof} The upper left arrow is given by inclusion and the upper
  right arrow given by restriction of the reduced bundle data to be
  non-trivial only on $Y_{\alpha}$. The arrows in the bottom row are
  defined accordingly.  Exactness in the middle of the top and the bottom
  row are immediate from the definitions.

To define the connecting homomorphisms on the left, consider an element in
$K^1_{\red}(Y_\alpha ;P).$ Choosing splittings as in \eqref{E-K-n.444} this
can be represented by a pair of bundles over $\widehat{B_\alpha}\times
Y_\alpha\times[0,1]$ with identifictions at all boundary hypersurfaces of
$Y_{\alpha},$ since these correspond to deeper strata, and at the ends of
the interval. Using the augmented pull-back map, this element lifts to a
pair of bundles over $\hB\times H_{\alpha}\times [0,1].$ Now, we can
identify $\widehat{B}\times H_{\alpha}\times [0,1]$ with a collar
neighborhood of $\widehat{B}\times H_{\alpha}$ in $\widehat{B}\times
Y$. Since the bundles are trivial over the ends of the interval, this
defines an element of $K^0_{\red}(Y;P\cup\{\alpha \})$ which is independent
of choices so defines a homomorphism. For the connecting homomorphism on
the right the construction is the same after tensoring with the Bott bundle
on $[0,1]^2$ and using one variable as the normal to the boundary and the
other as the suspension variable.

To check exactness at the top left corner, suppose an element of
$K^0_{\red}(Y;P)$ maps to the trivial element of
$K^0_{\red}(Y;P\cup\{\alpha\})$ under inclusion of reduced bundle
data. Then, for a stabilized representative $V_{\pm},$ inside $Y_{\alpha}$
there is a homotopy of the trivial bundle to itself (respecting triviality
of the bundle over deeper strata) that lifts to a homotopy from $V_{\pm}$
to the reduced bundle data corresponding to the trivial element. Such a
homotopy induces a pair of bundles over $\widehat{B}_{\alpha}\times
Y_{\alpha}\times [0,1]$ trivial at the endpoints and trivial at all strata
deeper than $Y_{\alpha}$ and hence an element of $K^1_{\red}
(Y_{\alpha};P).$ A similar argument shows that any element in the kernel of
the upper left arrow is of the form discussed in the construction of the
connecting homomorphism.

Finally we prove exactness at the bottom left corner. Suppose that an
element in $K^1_{\red}(Y_\alpha ;P)$ inserted into the neck near $\hB\times
H_{\alpha}$ is homotopic to the trivial element, the homotopy preserving
the appropriate trivializations corresponding to greated depth. This is
equivalent to the existence of a bundle $V_{[0,1]}$ over $\hB\times Y\times
[0,1]$ trivial at $\widehat{B}\times Y\times \{1\}$ and equal to the given
element at $\hB\times Y\times \{0\}$. Denote by $V_t$ the bundle above
$\hB\times Y\times \{t\};$ we need to show is that this data allows one to
extend the lift of the given bundle from $\hB\times H_{\alpha}\times [0,1]$
to $\B\times Y\times [0,1]$ such that the bundle data is trivial
everywhere (after a homotopy) except at $\widehat{B_{\alpha}}\times
Y_{\alpha}\times [0,1]$. Fixing a collar neighborhood
$\widehat{B}\times H_{\alpha}\times [0,1]\times [0,1] \subset
\widehat{B}\times Y\times [0,1]$ and extending the lifted bundle over
$\widehat{B}\times Y\times \{t\}$ by embedding it into the bundle
$V_0\big|_{\widehat{B}\times H_{\alpha}\times [t,1]}$. This can be seen to
be a bundle over $\widehat{B}\times Y\times [0,1]$ extending the lift and
be nullhomotopic due to the existence of the nullhomotopy in the
beginning. The only issue is that the bundle is not trivial over
$\widehat{B}\times Y\times \{0\}$. To fix this, we use the above
nullhomotopy to shift the bundle into the required one. Namely, over
$\widehat{B}\times Y\times \{t\}$ shift the bundle to
$V_{1-t}\big|_{\widehat{B}\times H_{\alpha}\times [t,1]}$. This is still a
bundle that now has the appropriate triviality conditions and it is still a
lift because the above homotopy does not affect $\widehat{B}\times
H_{\alpha}.$
\end{proof}

\begin{proposition}\label{10.7.2018.2} For any subtrees $A'=A\setminus P$
  and $A'\setminus\{\alpha\}$ there is a six term exact sequence 
\begin{equation}
\xymatrix{
H^0_{\dl}(Y_*;P\cup\{\alpha \})\ar[r]&H^0_{\dl}(Y_*;P)\ar[r]&H^0_{\dl}(Y_\alpha ;P)\ar[d]\\
H^1_{\dl}(Y_\alpha ;P)\ar[u]&H^1_{\dl}(Y_*;P)\ar[l]&H^1_{\dl}(Y_*;P\cup\{\alpha \}).\ar[l]
}
\label{10.7.2018.3a}\end{equation}
\end{proposition}

\begin{proof} This can be proved by combining standard arguments for the
  long exact sequence for the cohomology of a manifold relative to its
  boundary with the arguments as in the case of K-theory above.
\end{proof}

\section{The isomorphism}\label{Iso}

\begin{theorem}[See \cite{MR86g:55006}]\label{E-K-n.460} For the action of
  a compact abelian Lie group on a compact manifold the Chern
  character defines an isomorphism
\begin{equation}
K^*_G(X)\otimes\bbC\longrightarrow H_{G,\dl}^{*}(X).
\label{E-K-n.463}\end{equation}
\end{theorem}

\begin{proof} For any subtrees $A'=A\setminus P$ and $A'\setminus P',$ $P'=P\cup\{\alpha
  \}$ the exact sequences \eqref{10.7.2018.3} and \eqref{10.7.2018.3a}
  combine to form a commutative diagramme 
\begin{equation}
\xymatrix@!=3.5pc{
K^0_{\red}(Y_*;P')\ar[rr]\ar[dr]^{\Ch}&&
K^0_{\red}(Y_*;P)\ar[rr]\ar[d]^{\Ch}&&
K^0_{\red}(Y_\alpha ;P)\ar[dl]^{\Ch}\ar[ddd]
\\
&H^0_{\dl}(Y_*;P')\ar[r]&H^0_{\dl}(Y_*;P)\ar[r]&H^0_{\dl}(Y_\alpha ;P)\ar[d]\\
&H^1_{\dl}(Y_\alpha ;P)\ar[u]&H^1_{\dl}(Y_*;P)\ar[l]&H^1_{\dl}(Y_*;P').\ar[l]
\\
K^1_{\red}(Y_\alpha ;P)\ar[uuu]\ar[ur]_{\Ch}&&
K^1_{\red}(Y_*;P)\ar[ll]\ar[u]^{\Ch}&&
K^1_{\red}(Y_*;P' \})\ar[ul]^{\Ch}.\ar[ll]
}
\label{EKR.54}\end{equation}
Tensoring the K-theory part with $\bbC$ therefore also gives a commutative
diagramme with both six-term sequences exact. Now we may proceed by
induction using the fives lemma repeatedly 
\begin{equation}
A_0=\{0\}\subset A_1\dots\subset A_N=A
\label{EKR.38}\end{equation}
where the initial inductive step is given by Proposition~\ref{EKR.29}, for $A_0$ and
the central column is always exact.
\end{proof}

\section{Examples}\label{Ex}

The examples considered here are covered by Proposition 3.19 in Segal's
paper \cite{MR0234452}. We illustrate here how the same conclusions can be reached by
resolution.

Consider the standard circle action on the 2-sphere given by rotation
around an axis. The two poles are fixed points and on the complement the
action is free.  Radial blow up of the two poles replaces the 2-sphere by a
compact cylinder $I\times\bbS$ with free circle action and quotient, an interval, $I.$
Thus. from Theorem~\ref{abelianTh}, equivariant bundles up to isomorphism are in 1-1
correspondence with `reduced bundles' consisting of a bundle over $I$ with
isomorphisms with reduced bundles over $Y_*=\{I,\{N\},\{S\}\}.$ Over the
end-points these are simply (finitely-supported) bundles over
$\bbZ=\widehat{\UU(1)},$ i.e.\ a finite collection of vector spaces. Over
the principal space $I$ we simply have a vector bundle. The
pull-back map is trivial and the augmented pull-back map reduces to
summation. Thus reduced bundles in this case amount to a bundle over $I$
with decompostions into subspaces over the end-points. These decompositons
are unrelated up to isomorphism, except that dimensions must sum to the
dimension of the bundle over $I.$ Thus the equivariant K-theory in this
case is
\begin{equation}
K_{\UU(1)}(\bbS^2)=\cR(\UU(1))\oplus\cR(\UU(1))/\simeq
\label{E1}\end{equation}
where the relation is given by equality of the images of the dimension
maps. The action of $\widehat{\UU(1)}$ is the diagonal action on the
representation rings

The odd equivariant K-theory is given by the even equivariant K-theory of
$\bbS^2\times(0,1).$ The triviality of these bundles over the corners
$\{0\}\times\pa I$ and $\{1\}\times\pa I$ implies there triviality over the
end-poitns so the odd equivariant K-theory corresponds to pairs of bundles
supported on $(0,1)\times(I\setminus\pa I).$ Thus
$$
K^1_{\UU(1)}(\bbS^2)= K^0((-1,1)\times (0,1))=K^0(\bbR^2)=\bbZ.
$$

If $J$ is a generator of the odd K-theory and $H$ is the Hoof bundle then,
as an algebra, the total K-theory is generated by $H\oplus I=J.$ Since $H^2= 2H-J$
this recovers the result of Segal alluded to above.

\begin{lemma}\label{7.6.2018.1} For a product group action by compact
  abelian groups $A\subset G$ where $A$ acts trivially the equivariant
  K-theory is
\begin{equation}
K_{A\times G}(X)=\cR(A)\otimes K_{G}(X).
\label{7.6.2018.2}\end{equation}
\end{lemma}

\begin{proof} This follows immediately from the decomposition of bundles
  under the action of $A$ and the naturality of the lift of representations
  for a product.
\end{proof}

An immediate corollary of this Lemma in combination and the calculation
above shows it follows that for the rotation around on the sphere around an
axis with $n$ times the usual speed the equivariant K-theory is
\begin{equation}
\cR(\bbZ_n)
\otimes \bigl(\cI(\UU(1))\oplus\cI(\UU(1))/\simeq\bigr).
\label{E-K-n.407}\end{equation}

Consider next 2-dimensional complex projective space $\bbP^2$ with the circle
action 
\begin{equation}
\UU(1)\ni e^{i\theta}:\bbP^2\ni[z_1:z_2:z_3]\longmapsto [e^{i\theta}z_1:z_2:e^{-i\theta}z_3]\in
\bbP^2.
\label{E-K-n.393}\end{equation}
This is principally free and has three fixed points; at $[0:0:1],$
$[0:1:0]$ and $[1:0:0].$ On complement of the fixed points the isotropy
group is $\bbZ_2=\{1,\-1\}\subset\UU(1)$ on the the sphere $\bbP=\{[z_1:0:z_3]\}.$ 

This sphere is precisely the complement of the affine space given by
$z_2\not=0,$ i.e.\ $\{[z_1:1:z_3]\}$ around the isolated fixed point at
$[0:1:0].$ If we blow this point up \emph{in the complex sense} we replace
$\bbC^2$ by the canonical bundle over projective space 
\begin{equation}
[\bbC^2;\{0\}]_{\bbC}=K\bbP.
\label{E-K-n.396}\end{equation}

The complement of the zero section is covered by the two coordinate patches 
\begin{equation}
U'_1=\{(z_1,z_3);|z_1|>\ha|z_3|>0\},\ U'_3=\{(z_1,z_3);|z_3|>\ha|z_1|>0\}.
\label{E-K-n.400}\end{equation}
These project to dense subsets of the two regions of $\bbP^2:$ 
\begin{equation}
\begin{gathered}
U_1=\{[1:\zeta_2:\zeta_3)],\ \zeta_2=\frac1{z_1},\ \zeta_3=\frac{z_3}{z_1},\ |\zeta_3|<\ha\},
\\
U_3=\{[\eta_1:\eta_2:1)],\ \eta_2=\frac1{z_3},\ \eta_1=\frac{z_1}{z_3},\ |\eta_1|<\ha\}
\end{gathered}
\label{E-K-n.401}\end{equation}
which together cover $\bbP^2\setminus\{[0:1:0]\}.$ The real blow-up of
$\bbP,$ respectively $\{\zeta_2=0\}\subset U_1$ and $\{\eta_2=0\}\subset
U_3$ replaces them by 
\begin{equation}
\begin{gathered}
\widetilde U_1
=[0,\infty)_r\times\bbS\times\{|\zeta_3|<2\}\ni(r,e^{i\delta},\zeta_3)\longmapsto
  [1:re^{i\delta};\zeta_3]\in U_1\\
\widetilde U_3
=[0,\infty)_R\times\bbS\times\{|\zeta_1|<2\}\ni(R,e^{i\gamma},\zeta_1)\longmapsto
  [\zeta_1:Re^{i\gamma};1]\in U_3.
\end{gathered}
\label{E-K-n.402}\end{equation}
The intersection of $U_1$ and $U_3$ corresponds to the two regions and
transition map 
\begin{multline}
\widetilde U_1\supset\{(r,e^{i\delta},\zeta_3)\in [0,\infty)_r\times\bbS\times\bbC;\ha<|\zeta_3|<2\}\ni(r,e^{i\delta},\zeta_3)\longmapsto\\
(\frac r{|\zeta_3|},e^{i\delta}\frac{\overline{\zeta_3}}{|\zeta_3|},1/\zeta_1)\in
\widetilde U_1\supset\{(R,e^{i\gamma},\eta_1)\in [0,\infty)_R\times\bbS\times\bbC;\ha<|\eta_1|<2\}.
\label{E-K-n.403}\end{multline}
each $\widetilde U_i$ is the product of a boundary variable, a
circle and an open disk with the transition map patching this to the
product of a boundary variable and a circle bundle over the sphere.

The circle actions are therefore
\begin{equation}
\begin{gathered}
\widetilde U_1\ni(r,e^{i\delta},\zeta_3)\longmapsto (r,e^{i(\delta-\theta)},e^{-2i\theta}\zeta_3)
\\
\widetilde U_3\ni(R,e^{i\gamma},\eta_1)\longmapsto(R,e^{i(\delta+\theta)},e^{2i\theta}\eta_1).
\end{gathered}
\label{E-K-n.404}\end{equation}

The remaining fixed points at $[1;0;0]\in U_1$ and $[0:0:1]\in U_2$ lift
under the blow-up of $\bbP$ to the circles 
\begin{equation}
\{(0,e^{i\delta},0)\}\subset\widetilde U_1,\ \{(0,e^{i\gamma},0)\}\subset\widetilde U_3
\label{E-K-n.405}\end{equation}
on which the circle actions are now free. Nevertheless, our prescription
calls for these to be blown up.

Thus identifies
\begin{equation}
[\bbP^2;\{[0:1:0]\}]_{\bbC}=(\bbP K)\bbP
\label{E-K-n.394}\end{equation}
as projective compactification of the canonical bundle of $\bbP$ where the
exceptional divisor corresponding to the blow-up of $[0:1:0]$ is the zero
section and the section at infinity is the projective line at $\{z_2=0\}.$

The real blow-up of these two projective lines in \eqref{E-K-n.396} shows
that 
\begin{equation}
[\bbP^2;\{[0:1:0]\},\bbP]_{\bbR}=[0,1]\times\bbS^3\longrightarrow \bbP
\label{E-K-n.397}\end{equation}
is the corresponding bundle of cyclinders over $\bbP,$ so the product of a
radial invterval and the Hopf bundle $\bbS^3\longrightarrow \bbP.$  

The circle action on $\bbP^2$ restricted to the affine $\{z_2\not=0\}$ in
\eqref{E-K-n.397} is the action on the 3-sphere as 
\begin{equation}
(Z_1,Z_3)\longmapsto (e^{i\theta}Z_1,e^{-i\theta}Z_3)
\label{E-K-n.398}\end{equation}
which is the Hopf fibration after conjugating the second variable.

\providecommand{\bysame}{\leavevmode\hbox to3em{\hrulefill}\thinspace}
\providecommand{\MR}{\relax\ifhmode\unskip\space\fi MR }
\providecommand{\MRhref}[2]{%
  \href{http://www.ams.org/mathscinet-getitem?mr=#1}{#2}
}
\providecommand{\href}[2]{#2}


\end{document}